\documentclass{amsart}

\usepackage{amsmath}
\usepackage{amssymb}
\usepackage{indentfirst}
\usepackage{galois}
\usepackage{amsthm}
\usepackage{amsfonts}

\newtheorem{theorem}{Theorem}[section]
\newtheorem{proposition}[theorem]{Proposition}
\newtheorem{lemma}[theorem]{Lemma}
\newtheorem{corollary}[theorem]{Corollary}

\theoremstyle{definition}

\newtheorem{example}[theorem]{Example}
\newtheorem{problem}[theorem]{Problem}
\theoremstyle{remark}
\newtheorem{remark}[theorem]{Remark}

\numberwithin{equation}{section}

\begin{document}

\bibliographystyle{plain}
\title{Essential Spectra of Weighted Composition Operators Induced by Elliptic Automorphisms}

\author[X.T. Dong, Y.X. Gao and Z.H. Zhou] {Xing-Tang Dong, Yong-Xin Gao$^{*}$, and Ze-Hua Zhou}

\address{\newline  Xing-Tang Dong,\newline School of Mathematics, Tianjin University, Tianjin 300354, P.R. China.}

\email{dongxingtang@163.com}

\address{\newline  Yong-Xin Gao\newline  School of Mathematical Sciences and LPMC, Nankai University, Tianjin 300071, P.R. China}

\email{tqlgao@163.com}

\address{\newline Ze-Hua Zhou\newline School of Mathematics, Tianjin University, Tianjin 300354, P.R. China.}
\email{zehuazhoumath@aliyun.com; zhzhou@tju.edu.cn}


\date{}
\thanks{\noindent $^*$Corresponding author.\\
This work is supported in part by the National Natural Science
Foundation of China (Grant Nos. 11771323, 12001293); and the Natural Science Foundation of Tianjin City of China (Grant No. 19JCQNJC14700).}

\begin{abstract}
The spectrum of a weighted composition operator $C_{\psi, \varphi}$ who is induced by an automorphism has been investigated for over fifty years. However,  many results are got only under the condition that the weight function $\psi$ is continuous up to the boundary. In this paper we study the spectra and essential spectra of $C_{\psi,\varphi}$ on weighted Bergman spaces when $\varphi$ is an elliptic automorphism, without the assumption that $\psi$ is continuous up to the boundary.
\end{abstract}

\maketitle

\section{Introduction}

Let $D$ be the open unit disk in complex plane $\mathbb{C}$. And $H(D)$ denote the set of all analytic functions on $D$.  Given an analytic self-map $\varphi$ of $D$ and a function $\psi\in H(D)$, one can define a linear \textit{weighted composition operator} $C_{\psi,\varphi}$ on $H(D)$ by
$$C_{\psi,\varphi}f=\psi\cdot f\comp\varphi$$
for all $f\in H(D)$. By taking $\psi=1$ one get the \textit{composition operator} $C_\varphi$. And putting $\varphi$ be the identity map on $D$, gives the \textit{analytic multiplication} $M_\psi$.

The spectral properties of bounded weighted composition operators on kinds of analytic function spaces have been actively investigated over the past fifty years. In 1978, Kamowitz \cite{Kam} determined the spectrum of $C_{\psi,\varphi}$ on the disk algebra $A(D)=H(D)\cap C(\overline{D})$, when $\varphi$ is an automorphism of $D$. In the past ten years, several papers carried Kamowitz's project to different Banach spaces of analytic functions, such as Hardy space and Bergman space, see \cite{my,11,13}. However, the results of these papers have a fatal shortcoming: most of the results in \cite{my,11,13} are got under the condition that the weight function $\psi$ of the operator $C_{\psi,\varphi}$ belongs to $A(D)$, that is, $\psi$ is continuous up to the boundary of $D$. This condition is quite unnatural when we consider the property of $C_{\psi,\varphi}$ on a function space other than $A(D)$.

When $\psi$ is not continuous up to the boundary of $D$, the spectral properties of $C_{\psi,\varphi}$ can be much more complicated, and new methods  are needed to overcome the difficulties. Quite recently, Kitover and Orhon in \cite{19} give some descriptions about the spectrum of $C_{\psi,\varphi}$ on variety kinds of function spaces when $\varphi$ is a rotation, without requiring that $\psi\in A(D)$.

In this paper we investigate the spectrum and essential spectrum of weighted composition operator $C_{\psi,\varphi}$, when $\varphi$ is an elliptic automorphism, on weighted Bergman spaces without the assumption that $\psi\in A(D)$. It's an open question posed in \cite{19} (Problem 7.13 therein). For $\alpha>-1$ and $p\geqslant 1$, the weighted Bergman space $A_\alpha^p$ is a Banach space consisting of all analytic functions on $D$ such that
$$||f||^p=\int_D|f(z)|^pdA_\alpha(z)<\infty,$$
where the measure
$$dA_\alpha(z)=\frac{1+\alpha}{\pi}(1-|z|^2)^\alpha dA(z)$$
is normalized on $D$. Here $||f||$ is the norm of $f$ in $A_\alpha^p$. An elliptic automorphism is an biholomorphic map from $D$ to $D$ that has a unique fixed point in $D$. Note that if $\varphi$ is an elliptic automorphism, then the necessary and sufficient condition for $C_{\psi,\varphi}$ to be bounded on $A_\alpha^p$ is that $\psi$ is a bounded analytic function on $D$.

Our main results are in Section 5 and 6. In Section 5, we give a complete description of the spectrum and essential spectrum of $C_{\psi,\varphi}$ on weighted Bergman spaces when it is invertible. The spectrum and essential spectrum of a non-invertible weighted composition operator $C_{\psi,\varphi}$ is discussed in Section 6. We will show that the essential spectrum of $C_{\psi,\varphi}$ depends much on the location of the zeros of $\psi$ in the maximal ideal space of bounded analytic functions on $D$.

The main results of this paper are Corollary \ref{main1}, Corollary \ref{main2}, and Theorem \ref{main3}.

\section{Preliminaries}

In this section we will recall some basic facts about the maximal ideal space of bounded analytic functions on $D$, which is a key tool in our discussion through out this paper. Let $H^\infty$ denote the space of bounded analytic functions defined on $D$. Equipped with the supremum norm $\|\cdot\|_\infty$, $H^\infty$ is a Banach algebra.

For a Banach algebra $A$, the \textit{maximal ideal space} of $A$ is the collection of all non-zero multiplicative linear functionals on $A$,  equipped with the weak * topology induced from the dual space of $A$. We use $\mathfrak{M}_\infty$ to denote maximal ideal space of $H^\infty$. It is easy to check that $\mathfrak{M}_\infty$ is a compact Hausdorff space. The unit disk $D$ can be embedded homeomorphically into $\mathfrak{M}_\infty$, in a way that each point in $D$ is regarded as an evaluation functional.

The \textit{Gelfand transform} of a function $f\in H^\infty$, denoted by $\hat{f}$, is defined as
$$\hat{f}(m)=m(f) \quad, \quad\forall m\in\mathfrak{M}_\infty.$$
Obviously, $\hat{f}$ is a continuous extension of $f$ from $D$ to $\mathfrak{M}_\infty$. Thus, $H^\infty$ can be seen as a subalgebra of $C(\mathfrak{M}_\infty)$

Let $L^\infty(\partial D)$ be the space of $L^\infty$ functions on the unit circle $\partial D$. Then it is also a Banach algebra. We use $\mathfrak{M}_{L^\infty}$ to denote the maximal ideal space of $L^\infty(\partial D)$. It is well known that $\mathfrak{M}_{L^\infty}$ is exactly the Shilov boundary of $\mathfrak{M}_\infty$, that is, $\mathfrak{M}_{L^\infty}$ is the smallest close set contained in $\mathfrak{M}_\infty$ satisfying that
$$||f||_\infty=\sup_{m\in\mathfrak{M}_{L^\infty}}|f(m)|$$
for all $f\in H^\infty$.

The Gelfand transform of a function $f\in L^\infty(\partial D)$, also denoted by $\hat{f}$, is defined as
$$\hat{f}(m)=m(f) \quad, \quad\forall m\in\mathfrak{M}_{L^\infty}.$$
Similar with the case in $H^\infty$, for each $f\in L^\infty(\partial D)$, $\hat{f}$ is a continuous function on $\mathfrak{M}_{L^\infty}$. In the rest part of this paper, we shall always identify a function $f$ with its Gelfand transform $\hat{f}$ whenever $f$ belongs to either $H^\infty$ or $L^\infty(\partial D)$.

Since $H^\infty$ is a logmodular algebra, see Lemma \ref{log} in Section 4, according to Theorem 4.2 in \cite{BAF} there exists an unique positive Borel measure $\mu_0$ on $\mathfrak{M}_{L^\infty}$ such that $\mu_0(\mathfrak{M}_{L^\infty})=1$ and
$$f(0)=\int_{\mathfrak{M}_{L^\infty}}fd\mu_0$$
for all $f\in H^\infty$. $\mu_0$ is called the \textit{representing measure for the point $0$}. Moreover, for any $w\in D$ and $f\in H^\infty$, we have
\begin{align}\label{Poisson}
f(w)=\int_{\mathfrak{M}_{L^\infty}}fP_wd\mu_0,
\end{align}
where $P_w(\text{e}^{i\theta})=\frac{1-|w|^2}{|\text{e}^{i\theta}-w|^2}$ is the Poisson kernel for the point $w$. In fact, by the Riesz representation theorem, $\mu_0$ is determined by the linear functional $$g\mapsto\frac{1}{2\pi}\int_0^{2\pi}g(\text{e}^{i\theta})d\theta$$ for $g\in L^\infty(\partial D)$.

Suppose $f\in H^\infty$ is an inner function, then $f$ is unimodular in $L^\infty(\partial D)$, Therefore one have $|f(m)|=1$ for all $m\in\mathfrak{M}_{L^\infty}$. As a consequence, if $v$ is the outer part of a function $\psi\in H^\infty$, then $|v(m)|=|\psi(m)|$ for all $m\in\mathfrak{M}_{L^\infty}$. This fact will be used repeatedly in this paper.

For more information about the maximal ideal space of $H^\infty$, one can refer to Chapter \uppercase\expandafter{\romannumeral10} in \cite{BAF}. And some recent results on the topology structure of the maximal ideal space of $H^\infty$ can be found in \cite{Bru}.

\section{Notations}

The notations introduced in this section will be used throughout this paper.

We use $\mathbb{N}$ to denote the set of all nonnegative integers, and $\mathbb{Z}$ to denote the set of all integers.

$(H^\infty)^{-1}$ denote the collection of all invertible functions in $H^\infty$. That is, $$(H^\infty)^{-1}=\{f : f\in H^\infty , 1/f\in H^\infty\}.$$

Let $\varphi$ be an analytic self-map of the unit disc $D$. Then by $\varphi_n$ we mean the $n$-th iteration of the map $\varphi$, here $n\in\mathbb{N}$. Set $\varphi_0(z)=z$ be the identity. If $\varphi$ is an automorphism, then we denote by $\varphi_{-1}$ the inverse of $\varphi$, and $\varphi_{-n}$ is the $n$-th iteration of $\varphi_{-1}$ for $n\in\mathbb{N}$.

Suppose $\psi\in H^\infty$. For a given analytic self-map $\varphi$ of $D$, define $$\psi_{(n)}=\prod_{j=0}^{n-1}\psi\comp\varphi_j$$ for $n\in\mathbb{N}$. Set $\psi_{(0)}(z)=1$. Define $\rho_{\psi,\varphi}=\lim_{n\to\infty}||\psi_{(n)}||_\infty^{1/n}$.

For a weighted composition operator $C_{\psi,\varphi}$, we denote its the spectrum on $A^p_\alpha$ by $\sigma(C_{\psi,\varphi})$. Also the follows notations are used to denote the different parts of the spectrum of $C_{\psi,\varphi}$:

$\sigma_{e}(C_{\psi,\varphi})=\{\lambda\in\mathbb{C} : C_{\psi,\varphi}-\lambda \text{ is not Fredholm}\}$ is the essential spectrum;

$\sigma_{ap}(C_{\psi,\varphi})=\{\lambda\in\mathbb{C} : C_{\psi,\varphi}-\lambda \text{ is not bounded from below}\}$ is the approximate point spectrum;

$\sigma_{p}(C_{\psi,\varphi})=\{\lambda\in\mathbb{C} : C_{\psi,\varphi}-\lambda \text{ is not an injection}\}$ is the set of eigenvalues.

Each analytic self-map $\varphi$ of $D$ has a natural extension, also denoted by $\varphi$ in this paper, as a self-map of $\mathfrak{M}_\infty$, which is defined as follow:
$$\varphi(m)(f)=m(f\comp\varphi) \quad, \quad\forall f\in H^\infty, \forall m\in\mathfrak{M}_\infty.$$
When $\varphi(z)=\text{e}^{i\theta}z$ is a rotation, then for any $m\in\mathfrak{M}_\infty$ we will use $\text{e}^{i\theta}m$ to denote $\varphi(m)$. This is an abuse of notation, but it will cause no ambiguity in this paper.

Define $\text{e}^{i\theta}E=\{\text{e}^{i\theta}x : x\in E\}$, here $E$ is a subset of either $\mathfrak{M}_{\infty}$ or $\partial D$.

\section{Spectral Radius}

Let $\varphi$ be an elliptic automorphism of $D$. The \textit{order} of $\varphi$ is the smallest positive integer $n$ such that $\varphi_n(z)=z$ is the identity on $D$. If such $n$ does not exist, then we say $\varphi$ is of order $\infty$. When the order of $\varphi$ is finite, say $n_0\in\mathbb{N}$, then $C_{\psi,\varphi}^{n_0}=M_{\psi_{(n_0)}}$ is an analytic multiplication. So in this case, it is not difficult to figure out the spectrum of $C_{\psi,\varphi}$. For a complete discussion, see \cite{my} for example.

In this section we will investigate the spectral radius of $C_{\psi,\varphi}$ where $\varphi$ is an elliptic automorphism of order $\infty$. Assume that $a\in D$ is the fixed point of $\varphi$. Let $\varphi_a(z)=\frac{a-z}{1-\overline{a}z}$ be the involution automorphism that exchanges $0$ and $a$. Then $\tau=\varphi_a\comp\varphi\comp\varphi_a$ is an automorphism fixing the point $0$, hence a rotation. A simple calculation shows that
\begin{align}\label{xs}
C_{\varphi_a}^{-1}C_{\psi,\varphi}C_{\varphi_a}=C_{\psi\comp\varphi_a,\tau}.
\end{align}
So $C_{\psi,\varphi}$ is similar to the weighted composition operator $C_{\psi\comp\varphi_a,\tau}$, which means that they have the same
spectrum. Note that a rotation is of order $\infty$ if and only if it is an irrational one. Therefore in the rest part of this paper, we shall just handle the case when $\varphi$ is an irrational rotation.

It is easy to check that
$$C_{\psi,\varphi}^n=C_{\psi_{(n)},\varphi_n}.$$
Since $C_{\varphi_n}$ is unitary on $A_\alpha^p$ when $\varphi$ is a rotation, the spectral radius of $C_{\psi,\varphi}$ is $$\rho_{\psi,\varphi}=\lim_{n\to\infty}||\psi_{(n)}||_\infty^{1/n}.$$

The following theorem, known as the Birkhoff's Ergodic Theorem, is an useful tool in our discussion.

\begin{theorem}\label{erg1}
Suppose $(X,\mathfrak{F},\mu)$ is a probability space. Let $T$ be a surjective map from $X$ onto itself such that for any $A\in\mathfrak{F}$, we have $T^{-1}A\in\mathfrak{F}$ and $\mu(T^{-1}A)=\mu(A)$. If $T$ is ergodic, then for any $f\in L^1(X,\mu)$ we have
$$\lim_{n\to\infty}\frac{1}{n}\sum_{k=1}^{n}f(T^kx)=\int_Xfd\mu$$
for $\mu$-almost every $x\in X$.
\end{theorem}
Here by saying a map $T$ is ergodic we mean that $T^{-1}A=A$ implies $\mu(A)$ is $0$ or $1$ for all $A\in\mathfrak{F}$. For more information about this theorem, one can turn to \cite{Erg}.

The unit circle $\partial{D}$, along with the Lebesgue measure $d\theta/2\pi$, is a probability space. If $\varphi$ is a rotation, then $\varphi$ can be seen as a surjective map from $\partial\mathbb{D}$ onto itself. It is well known that the rotation $\varphi$ is ergodic on $\partial\mathbb{D}$ if and only if it is an irrational one.

Suppose $\psi\in H^\infty$ is not identically zero, then $\log|\psi|$ belongs to $L^1(\partial\mathbb{D},\frac{d\theta}{2\pi})$, see Theorem 2.7.1 in \cite{GTM237}. So by Ergodic Theorem, if $\varphi$ is an irrational rotation, then
$$\lim_{n\to\infty}\frac{1}{n}\sum_{k=1}^{n}\log |\psi(\varphi_k(\zeta))|=\frac{1}{2\pi}\int_0^{2\pi}\log |\psi(e^{i\theta})|d\theta$$
for almost every $\zeta\in\partial\mathbb{D}$. Or equivalently we have
$$\lim_{n\to\infty}|\psi_{(n)}(\zeta)|^{1/n}=\exp\left(\frac{1}{2\pi}\int_0^{2\pi}\log|\psi(e^{i\theta})|d\theta\right)$$
for almost every $\zeta\in\partial\mathbb{D}$. Therefore, if we write $v$ for the outer part of the function $\psi$, then the spectral radius of $C_{\psi,\varphi}$ is no less than
$$\exp\left(\frac{1}{2\pi}\int_0^{2\pi}\log|\psi(e^{i\theta})|d\theta\right)=|v(0)|.$$

For the general case, if $\varphi$ is an elliptic automorphism with fixed point $a\in D\backslash\{0\}$ of order $\infty$, then using (\ref{xs}) one can conclude that the spectral radius of  $C_{\psi,\varphi}$ is no less than $|v(a)|$, where $v$ is the outer part of $\psi$. Thus we have proved the following proposition.

\begin{proposition}\label{zc}
Suppose $\varphi$ is an elliptic automorphism of order $\infty$ and $\psi\in H^\infty$. Then the spectral radius of $C_{\psi,\varphi}$ on $A_\alpha^p$ is $\rho_{\psi,\varphi}$.

If we take $v$ to be the outer part of function $\psi$, and assume that $a\in D$ is the fixed point of $\varphi$, then $\rho_{\psi,\varphi}\geqslant|v(a)|$.
\end{proposition}

It has been shown in \cite{11} and \cite{my} that if $\psi\in A(D)$, then the spectral radius of $C_{\psi,\varphi}$ is exactly $|v(a)|$. However, one can never expect $\rho_{\psi,\varphi}=|v(a)|$ holds for all the cases, see the following Example \ref{ex1}.

Recall that $H^\infty$ is a \textit{logmodular subalgebra} of $L^\infty(\partial D)$, which means that the set
$$\{\log|f| : f\in (H^\infty)^{-1}\}$$
is dense in $L^\infty(\partial D)$. In fact, we have the following lemma, which is Theorem 4.5 in \cite{BAF}.

\begin{lemma}\label{log}
Every real-valued function $g(t)$ in $L^\infty(\partial D)$ has the form $\log |f(t)|$, where $f\in (H^\infty)^{-1}$.
\end{lemma}

\begin{remark}\label{logre}
Suppose $g(t)$ is a real-valued function in $L^1(\partial D)$ that is bounded above, then a similar discussion shows that $g(t)$ is of the form $\log |f(t)|$ for some $f\in H^\infty$. See page 53 of \cite{Hoffman2}.
\end{remark}

\begin{example}\label{ex1}
According to Rudin \cite{Rudin}, a Borel set $U$ in $\partial D$ is called \textit{permanently positive} if
$$\frac{d\theta}{2\pi}(\bigcap_{j=1}^n\zeta_jU)>0$$
for any $n\in\mathbb{N}$ and any choice of $\{\zeta_j\}_{j=1}^{n}$ on $\partial D$. Now suppose $U$ is a permanently positive set with $\frac{d\theta}{2\pi}(U)<1$, and $\chi_U\subset L^\infty(\partial D)$ is the characteristic function for $U$. Define $\widehat{U}=\{m\in\mathfrak{M}_{L^\infty} : \chi_U(m)=1\}$, then $\widehat{U}$ is a clopen subset of $\mathfrak{M}_{L^\infty}$. For any finite set $\{\zeta_1 , \zeta_2 , ... , \zeta_n\}$ in $\partial D$, $\bigcap_{j=1}^n\zeta_jU$ has positive Lebesgue measure, so the set
$$\bigcap_{j=1}^n\zeta_j\widehat{U}=\widehat{\bigcap_{j=1}^n\zeta_jU}$$
is non-void. Note that $\zeta\widehat{U}$ is closed in $\mathfrak{M}_{L^\infty}$ for each $\zeta\in\partial D$, thus by the finite intersection property we can conclude that
$$\bigcap_{\zeta\in\partial D}\zeta \widehat{U}\ne\emptyset.$$
Take $m\in\bigcap_{\zeta\in\partial D}\zeta \widehat{U}$, which means that $\zeta m\in\widehat{U}$ for all $\zeta\in\partial D$. Now if we take a function $\psi\in (H^\infty)^{-1}$ such that $\log|\psi|$ equals to $\chi_G$ in $L^\infty(\partial D)$, then $|\psi(\zeta m)|=\text{e}$ for all $\zeta\in\partial D$. So for any irrational rotation $\varphi(z)=\eta z$, we have $$\rho_{\psi,\varphi}\geqslant\lim_{n\to\infty}\prod_{j=0}^{n-1}|\psi(\eta^jm)|^{1/n}=\text{e}.$$
On the other hand, we know $\rho_{\psi,\varphi}\leqslant||\psi||_\infty=\text{e}.$
Therefore the spectral radius of $C_{\psi,\varphi}$ on $A_\alpha^p$ is $\rho_{\psi,\varphi}=\text{e}$, which is large than $|\psi(0)|=\text{e}^{\frac{d\theta}{2\pi}(U)}.$
\end{example}

In \cite{Russion} Kitover proves that for any irrational rotation $\varphi$ and $\psi\in H^\infty$,
$$\rho_{\psi,\varphi}=\max_{\mu\in M_\varphi}\exp\left(\int_{\mathfrak{M}_{L^\infty}}\log|\psi|d\mu\right),$$
where $M_\varphi$ is the set of all $\varphi$-invariant regular probability Borel measures on $\mathfrak{M}_{L^\infty}$. However, this is not a explicit description of $\rho_{\psi,\varphi}$. In some special cases, $\rho_{\psi,\varphi}$ can have a more clear expression. The next lemma is Theorem 4.2 in \cite{19}.

\begin{lemma}\label{ty}
Suppose $\varphi$ is an irrational rotation and $f$ is an upper semi-continuous function on $\partial\mathbb{D}$, then
$$\lim_{n\to\infty}\sup_{\zeta\in\partial D}\left(\frac{1}{n}\sum_{j=0}^{n-1}f\comp\varphi_j(\zeta)\right)=\frac{1}{2\pi}\int_{0}^{2\pi}f(\text{e}^{i\theta})d\theta.$$
\end{lemma}

Let $v\in H^\infty$ be an outer function.  Then according to Remark \ref{logre}, we can define $v^*$ to be the outer function such that
$$|v^*(\text{e}^{i\theta})|=\limsup_{z\to\text{e}^{i\theta}}|v(z)|$$
a.e. on $\partial D$. Then $v^*$ is unique up to multiplication by a constant of modulus one. Similarly, if $g(\text{e}^{i\theta})=\liminf_{z\to\text{e}^{i\theta}}\log|v(z)|$ is in $L^\infty(\partial D)$, we can define $v_*$ to be the outer function such that
$$|v_*(\text{e}^{i\theta})|=\liminf_{z\to\text{e}^{i\theta}}|v(z)|$$
a.e. on $\partial D$. Otherwise, we shall just set $v_*=0$. Denote $v_{**}=((v_*)^*)_*$. Now we will improve Proposition \ref{zc} as follows.

\begin{proposition}\label{rad}
Suppose $\varphi$ is an elliptic automorphism of order $\infty$ with fixed point $a\in D$ and $\psi\in H^\infty$. Let $v$ be the outer part of $\psi$. Then
$$\max\left\{|v(a)| , |v_{**}(a)|\right\}\leqslant\rho_{\psi,\varphi}\leqslant|v^*(a)|.$$
\end{proposition}

\begin{proof}
We will just proof the situation where $a=0$, or equivalently, $\varphi$ is an irrational rotation. Then the general case follows directly from (\ref{xs}).

By the definition of $v^*$, $\log |v^*|$ equals to an upper semi-continuous function a.e. on $\partial D$. Moreover, since the inner part of $\psi$ take a value of modulus one at each point in $\mathfrak{M}_{L^\infty}$, we have $|\psi(m)|\leqslant|v^*(m)|$ for all $m\in\mathfrak{M}_{L^\infty}$. Therefore by Lemma \ref{ty} we have
\begin{align*}
\rho_{\psi,\varphi}\leqslant\rho_{v*,\varphi}&=\exp\left(\frac{1}{2\pi}\int_{0}^{2\pi}\log|v^*(\text{e}^{i\theta})|d\theta\right)=|v^*(0)|.
\end{align*}

Now let us proof that $|v_{**}(0)|\leqslant\rho_{\psi,\varphi}$. Fix any $\epsilon>0$. By the definition of $v_*$, for almost every $\zeta\in\partial D$ there exist an open arc $I_\zeta$ center at such that $$\log|(v_*)^*(\zeta')|>\log|v_{**}(\zeta)|-\epsilon\quad,\quad \forall\zeta'\in I_\zeta.$$
As a consequence, the set
$$E_\zeta=\{\zeta'\in I_\zeta : \log|v_*(\zeta')|>\log|v_{**}(\zeta)|-\epsilon\}$$
is dense in $I_\zeta$. Moreover, since $|v_*|$ equals to a lower semi-continuous function on $\partial D$, we may assume $E_\zeta$ to be open after modifying the values of $v_*$ on a set of Lebesgue measure zero. So for any finite set of points $\{\zeta_j\}_{j=1}^n$ on $\partial D$, we have
\begin{align}\label{now}
\frac{d\theta}{2\pi}\left(\bigcap_{j=1}^n\overline{\zeta}E_\zeta\right)>0.
\end{align}
Let
$$\widehat{E}_\zeta=\{m\in\mathfrak{M}_{L^\infty} : \log|v_*(m)|\geqslant\log|v_{**}(\zeta)|-\epsilon\}.$$
Then (\ref{now}) implies that $\bigcap_{j=1}^n\overline{\zeta}\widehat{E}_\zeta$ is not empty. Since each $\widehat{E}_\zeta$ is closed in $\mathfrak{M}_{L^\infty}$, by the finite intersection property we can know that
$$\bigcap_{\zeta\in\partial D}\overline{\zeta}\widehat{E}_\zeta\ne\emptyset.$$
Take $m_0\in\bigcap_{\zeta\in\partial D}\overline{\zeta}\widehat{E}_\zeta$, then for all $\theta\in[0,2\pi)$ we have
$$\log|v(\text{e}^{i\theta}m_0)|\geqslant\log|v_*(\text{e}^{i\theta}m_0)|\geqslant\log|v_{**}(\text{e}^{i\theta})|-\epsilon.$$
Hence by Ergodic Theorem, we have
\begin{align*}
\rho_{\psi,\varphi}&\geqslant\exp\left(\frac{1}{2\pi}\int_0^{2\pi}\log|v(\text{e}^{i\theta}m_0)|\right)\\
&\geqslant\exp\left(\frac{1}{2\pi}\int_0^{2\pi}\log|v_{**}(\text{e}^{i\theta})|-\epsilon\right)\\
&=|v_{**}(0)|\cdot\text{e}^{-\epsilon}
\end{align*}
Letting $\epsilon\to 0$, we have $\rho_{\psi,\varphi}\geqslant|v_{**}(0)|$.
\end{proof}

\begin{example}
Let $G$ be an open dense subset of $\partial D$ with $\frac{d\theta}{2\pi}(G)<1$. And $\chi_G\subset L^\infty(\partial D)$ is the characteristic function for $G$. We now take a function $\psi\in (H^\infty)^{-1}$ such that $\log|\psi|$ equals to $\chi_G$ in $L^\infty(\partial D)$, and consider the spectrum of $C_{\psi,\varphi}$, where $\varphi$ is an irrational rotation, on $A_\alpha^p$.

Firstly, note that the open dense $G$ is apparently permanently positive. So Example \ref{ex1} shows directly that $\rho_{\psi,\varphi}=\text{e}$.

But here let's discuss this example by using Proposition \ref{rad} other than Example \ref{ex1}. Since $\psi$ is outer in this case, it is easy to see that $\psi^*=\psi_{**}=\text{e}$. So by Proposition \ref{rad}, the spectral radius of $C_{\psi,\varphi}$ is $\rho_{\psi,\varphi}=\text{e}$. On the other hand, a simple calculation shows that $C_{\psi,\varphi}^{-1}=C_{\tilde{\psi},\varphi_{-1}}$ where $\tilde{\psi}=\frac{1}{\psi\comp\varphi_{-1}}$. Since $G$ is open, one has $\tilde{\psi}^*=\tilde{\psi}$. So the spectral radius of $C_{\psi,\varphi}^{-1}$ is $\rho_{\tilde{\psi},\varphi}=\frac{1}{\psi(0)}=\text{e}^{-\frac{d\theta}{2\pi}(G)}$.

According to the discussion in Section 4, one in fact has
$$\sigma(C_{\psi,\varphi})=\sigma_e(C_{\psi,\varphi})=\{\lambda\in\mathbb{C} : \text{e}^{\frac{d\theta}{2\pi}(G)}\leqslant|\lambda|\leqslant\text{e}\}.$$
This example can also be found in \cite{19}, but it can actually be derived from the proof of Theorem 1.1 in \cite{Orb}.
\end{example}

\begin{remark}\label{mu}
Another useful model for Ergodic Theorem involves in our paper is $\mathfrak{M}_{L^\infty}$ with measure $\mu_0$. It is easy to see that each irrational rotation $\varphi$ is ergodic on $\mathfrak{M}_{L^\infty}$ with respect to the measure $\mu_0$, see \cite{Orb}. So by Ergodic Theorem, for any $\psi\in H^\infty$, the following equation
$$\lim_{n\to\infty}|\psi_{(n)}(m)|^{1/n}=\exp\left(\int_{\mathfrak{M}_{L^\infty}}\log|\psi|d\mu_0\right)=|v(0)|$$
holds for $\mu_0$-almost every $m\in\mathfrak{M}_{L^\infty}.$ Here $v$ is the outer part of $\psi$.
\end{remark}

\section{Spectra of Invertible Operators}

In this section, we will give a complete description of the spectrum and essential spectrum of invertible operator $C_{\psi, \varphi}$ when $\varphi$ is an elliptic automorphism of order $\infty$. Note that when $\varphi$ is an elliptic automorphism, $C_{\psi, \varphi}$ is invertible if and only if $\psi\in (H^\infty)^{-1}$, i.e., $\psi$ has no zero in $\mathfrak{M}_{\infty}$.

The proof of next lemma is the same as the proof of Proposition 7.11 in \cite{19}.
\begin{lemma}\label{ri}
Suppose $\varphi$ is an irrational rotation and $\psi\in H^\infty$. Let $\sigma_e(C_{\psi, \varphi})$ be the essential spectrum of $C_{\psi,\varphi}$ on $A_\alpha^p$. If $\lambda\in\sigma_e(C_{\psi, \varphi})$, then for any $\theta\in\mathbb{R}$ we have $\mathrm{e}^{i\theta}\lambda\in\sigma_e(C_{\psi, \varphi})$.
\end{lemma}

The next lemma is crucial for our discussion in this section. The proof is based on Remark \ref{mu}.

\begin{lemma}\label{ts1}
Suppose $\varphi$ is an irrational rotation and $\psi\in H^\infty$. Let $v$ be the outer part of $\psi$. If $\rho_{\psi, \varphi}=\lim_{n\to\infty}||\psi_{(n)}||_\infty^{1/n}$ is greater than $|v(0)|$, then for any $\epsilon>0$ and $n\in\mathbb{N}$, there exist $m_0\in\mathfrak{M}_{L^\infty}$ and $n'>n$ such that
$$|\psi_{(n)}(m_0)|^{\frac{1}{n}}>\rho_{\psi, \varphi}-\epsilon$$
and
$$|\psi_{(n')}\comp\varphi_{-n'}(m_0)|^{\frac{1}{n'}}<|v(0)|+\epsilon.$$

\end{lemma}

\begin{proof}
According to Remark \ref{mu}, the equation
\begin{align}\label{jj}
\lim_{k\to\infty}\left|\prod_{j=1}^{k}\psi\comp\varphi_{-j}(m)\right|^{1/k}=|v(0)|
\end{align}
holds for $\mu_0$-almost every $m$ in $\mathfrak{M}_{L^\infty}$. By the definition of $\rho_{\psi, \varphi}$, for any $\epsilon>0$ and $n\in\mathbb{N}$, the set
$$J=\{m\in\mathfrak{M}_{L^\infty} : |\psi_{(n)}(m)|>(\rho_{\psi, \varphi}-\epsilon)^n\}$$
is a non-empty open set in $\mathfrak{M}_{L^\infty}$. Moreover, since $||\psi_{(n)}||_\infty\geqslant\rho_{\psi, \varphi}^n$, by (\ref{Poisson}) we have $\mu_0(J)> 0$. Therefore we can pick a point $m_0\in J$ such that (\ref{jj}) holds for $m_0$. Thus for $n'\in\mathbb{N}$ large enough, we have
$$|\psi_{(n')}\comp\varphi_{-n'}(m_0)|^{\frac{1}{n'}}=\left|\prod_{j=1}^{n'}\psi\comp\varphi_{-j}(m_0)\right|^{1/n'}<|v(0)|+\epsilon.$$
So $m_0$ is the point we want.
\end{proof}

We now prove our main theorem in this section, with the help of Lemma \ref{ts1}.

\begin{theorem}\label{inver}
Suppose $\varphi$ is an irrational rotation and $\psi\in H^\infty$. Let $v$ be the outer part of $\psi$. If $\rho_{\psi,\varphi}$ is greater than $|v(0)|$, then the set
$$\{\lambda\in\mathbb{C} : |v(0)|\leqslant|\lambda|\leqslant \rho_{\psi,\varphi}\}$$
is contained in the essential spectrum of $C_{\psi,\varphi}$ on $A_\alpha^p$.
\end{theorem}

\begin{proof}
First fix an arbitrary positive number $\lambda$ in the interior of the target set, and take $q_1,q_2>0$ such that $|v(0)|<q_1<\lambda<q_2<\rho_{\psi, \varphi}$.

We now claim that for each $k\in\mathbb{N}$, we can find $n_k>2k$ and $g_k\in A^p_\alpha$ such that
\begin{align}\label{1by}
\left|\left|\psi_{(k)}\cdot g_k\right|\right|>q_2^{k}||g_k||
\end{align}
and
\begin{align}\label{2by}
\left|\left|\psi_{(n_k)}\cdot g_k\comp\varphi_{n_k-k}\right|\right|<q_1^{n_k-k}\left|\left|\psi_{(k)}g_k\right|\right|.
\end{align}

In fact, by Lemma \ref{ts1}, for any fixed $k\in\mathbb{N}$ and any $\epsilon>0$ there exist $m\in\mathfrak{M}_{L^\infty}$ and $n_k>2k$ such that
$$\left|\psi_{(k)}(m)\right|^{\frac{1}{k}}=l_1>\rho_{\psi, \varphi}-\epsilon$$
and
$$\left|\psi_{(n_k-k)}\comp\varphi_{k-n_k}(m)\right|^{\frac{1}{n_k-k}}=l_2<|v(0)|+\epsilon.$$
Let
$$h_k=(1+\epsilon)\cdot\min\left\{\frac{|\psi_{(k)}|}{l_1^k} , \frac{l_2^{n_k-k}}{|\psi_{(n_k-k)}\comp\varphi_{k-n_k}|}\right\},$$
then $h_k$ is a continuous function on $\mathfrak{M}_{L^\infty}$, and $h_k(m)=1+\epsilon$. Moreover, since
$$\int_{\mathfrak{M}_{L^\infty}}\log h_kd\mu_0>-\infty,$$
by Remark \ref{logre} we can take outer function $\tilde{h}_k\in H^\infty\subset A_\alpha^p$ such that $|\tilde{h}_k|=h_k$ on $\mathfrak{M}_{L^\infty}$. Let $\tilde{g}_k=\phi_{(k)}\tilde{h}_k$ where $\phi=\psi/v$ is the inner part of $\psi$. Then
$$||\tilde{g}_k||_\infty\geqslant|\tilde{g}_k(m)|=h_k(m)=1+\epsilon>1.$$

Now let $E_k=\{z\in D : |\tilde{g}_k|\geqslant1\}$. Notice that
\begin{align*}
E_k\subset\left\{z : |\psi_{(k)}(z)|\geqslant\frac{l_1^k}{1+\epsilon}\right\}\bigcap\left\{z : |\psi_{(n_k-k)}\comp\varphi_{k-n_k}(z)|\leqslant(1+\epsilon)l_2^{n_k-k}\right\}.
\end{align*}
So by the Lebesgue Dominated Convergence Theorem, for $N\in\mathbb{N}$ large enough we have
\begin{align*}
\left|\left|\psi_{(k)}\cdot \tilde{g}_k^N\right|\right|&\geqslant\left(\int_{E_k}\left|\psi_{(k)}\tilde{g}_k^N\right|^pdA_\alpha\right)^{1/p}\\
&\geqslant\frac{l_1^{k}}{1+\epsilon}\left(\int_{E_k}|\tilde{g}_k|^{pN}dA_\alpha\right)^{1/p}\\
&\geqslant\frac{1-\epsilon}{1+\epsilon}\cdot(\rho_{\psi,\varphi}-\epsilon)^{k}\left|\left|\tilde{g}_k^N\right|\right|,
\end{align*}
and
\begin{align*}
\left|\left|\psi_{(n_k)}\cdot \tilde{g}_k^N\comp\varphi_{n_k-k}\right|\right|&=\left|\left|\psi_{(n_k-k)}\comp\varphi_{k-n_k}\cdot\psi_{(k)}\cdot \tilde{g}_k^N\right|\right|\\
&\leqslant\frac{1}{1-\epsilon}\left(\int_{E_k}\left|\psi_{(n_k-k)}\comp\varphi_{k-n_k}\cdot\psi_{(k)}\cdot \tilde{g}_k^N\right|^pdA_\alpha\right)^{1/p}\\
&\leqslant\frac{1+\epsilon}{1-\epsilon}\cdot l_2^{n_k-k}\left(\int_{E_k}\left|\psi_{(k)}\tilde{g}_k^N\right|^pdA_\alpha\right)^{1/p}\\
&\leqslant \frac{1+\epsilon}{1-\epsilon}\cdot (|v(0)|+\epsilon)^{n_k-k}\left|\left|\psi_{(k)}\tilde{g}_k^N\right|\right|.
\end{align*}
By taking $\epsilon>0$ sufficiently small and taking $g_k=\tilde{g}_k^N$, we have (\ref{1by}) and (\ref{2by}) hold.

Let
$$f_k=\sum_{j=0}^{n_k-1}\lambda^{k-j-1}\psi_{(j)}\cdot g_k\comp\varphi_{j-k}.$$
Then
$$\left(C_{\psi,\varphi}-\lambda\right)f_k=\frac{\psi_{(n_k)}\cdot g_k\comp\varphi_{n_k-k}}{\lambda^{n_k-k}}-\lambda^k g_k\comp\varphi_{-k}.$$
So
\begin{align}\label{ti}
||\left(C_{\psi,\varphi}-\lambda\right)f_k||&\leqslant\frac{||\psi_{(n_k)}\cdot g_k\comp\varphi_{n_k-k}||}{\lambda^{n_k-k}}+\lambda^k ||g_k\comp\varphi_{-k}||\nonumber\\
&\leqslant\left(\frac{q_1}{\lambda}\right)^{n_k-k}\left|\left|\psi_{(k)}g_k\right|\right|+\left(\frac{\lambda}{q_2}\right)^{k}\left|\left|\psi_{(k)}g_k\right|\right|.
\end{align}
Here the last inequality follows from (\ref{1by}) and (\ref{2by}).

Let $\lambda_{k,s}=\lambda \text{e}^{\frac{2s\pi i}{n_k}}$ for $s=1,2,3,...,n_k$. Replacing $\lambda$ in the definition of each function $f_k$ by $\{\lambda_{k,s}\}_{s=1}^{n_k}$, we get $n_k$ functions, denoted by $\{f_{k,s}\}_{s=1}^{n_k}$ respectively. It is easy to check that
$$\sum_{s=1}^{n_k}\lambda_{k,s} f_{k,s}=n_k\cdot\psi_{(k)}\cdot g_k,$$
so there exist $s_k\in\{1,2,3,...,n_k\}$ such that $\lambda||f_{k,s_k}||\geqslant||\psi_{(k)}g_k||.$ However, (\ref{ti}) implies that
$$\lim_{k\to\infty}\frac{||\left(C_{\psi,\varphi}-\lambda_{k,s_k}\right)f_{k,s_k}||}{\left|\left|\psi_{(k)}g_k\right|\right|}=0.$$
By passing to a subsequence, we may assume that the sequence $\{\lambda_{k,s_k}\}_{k=1}^{\infty}$ convergences to a point $\lambda_0$. Then
\begin{align*}
\lim_{k\to\infty}\frac{||\left(C_{\psi,\varphi}-\lambda_0\right)f_{k,s_k}||}{||f_{k,s_k}||}&=\lim_{k\to\infty}\frac{||\left(C_{\psi,\varphi}-\lambda_{k,s_k}\right)f_{k,s_k}||}{||f_{k,s_k}||}\\
&\leqslant\lambda\cdot\lim_{k\to\infty}\frac{||\left(C_{\psi,\varphi}-\lambda_{k,s_k}\right)f_{k,s_k}||}{\left|\left|\psi_{(k)}g_k\right|\right|}\\
&=0.
\end{align*}
This means that the operator $C_{\psi,\varphi}-\lambda_0$ is not bounded from below on $A_\alpha^p$. Now we want to show that $\lambda_0$ is not an eigenvalue of $C_{\psi,\varphi}$. If this is true, then $\lambda_0$ belongs to $\sigma_{ap}(C_{\psi,\varphi})\backslash\sigma_{p}(C_{\psi,\varphi})\subset\sigma_{e}(C_{\psi,\varphi})$, and by Lemma \ref{ri} we can get our conclusion.

To this end, let's assume that there exist $f\in A_\alpha^p\backslash\{0\}$ such that $C_{\psi,\varphi}f=\lambda_0f$. Suppose $f=\sum_{j=K}^\infty a_jz^j$ where $a_K\ne 0$. Then
$$\psi\cdot\sum_{j=K}^\infty a_j\varphi^j=\sum_{j=K}^\infty \lambda_0a_jz^j.$$
Taking $K$-th derivative at zero to both sides of this equation we can see that $|\lambda_0|=|\psi(0)|$. But this is impossible since $$|\lambda_0|=\lambda>|v(0)|\geqslant|\psi(0)|.$$
\end{proof}

Now we can give our final result in this section as follows.

\begin{corollary}\label{main1}
Suppose $\psi\in (H^\infty)^{-1}$ and $\varphi$ is an elliptic automorphism of order $\infty$. Then on $A_\alpha^p$ one has
$$\sigma(C_{\psi,\varphi})=\sigma_e(C_{\psi,\varphi})=\{\lambda\in\mathbb{C} : \rho_{\frac{1}{\psi},\varphi}^{-1}\leqslant|\lambda|\leqslant \rho_{\psi,\varphi}\}.$$
\end{corollary}

\begin{proof}
If $\rho_{\frac{1}{\psi},\varphi}^{-1}=\rho_{\psi,\varphi}$, then the result follows directly from Lemma \ref{ri}.

If $\rho_{\frac{1}{\psi},\varphi}^{-1}<\rho_{\psi,\varphi}$, then the result is a combination of Theorem \ref{inver} and the fact that $C_{\psi,\varphi}$ is invertible with $C_{\psi,\varphi}^{-1}=C_{\frac{1}{\psi\comp\varphi_{-1}},\varphi_{-1}}$.
\end{proof}

\section{Spectra of Non-invertible Operators}

In this section, we will discuss the (essential) spectrum of $C_{\psi,\varphi}$ when it is not invertible. Since $\varphi$ is an automorphism, if $C_{\psi,\varphi}$ is not invertible, then $\psi$ must have zeros in $\mathfrak{M}_{\infty}$. It turns out that the essential spectrum of $C_{\psi,\varphi}$ depends much on the location of the zeros of $\psi$.

First let us treat the cases when $\psi$ has zeros in $\mathfrak{M}_{L^\infty}$. This is equivalent to the condition that the outer part of $\psi$ dose not belongs to $(H^\infty)^{-1}$. The following two results, Lemma \ref{ts2} and Theorem \ref{noninver}, are parallel to Lemma \ref{ts1} and Theorem \ref{inver} respectively.

\begin{lemma}\label{ts2}
Suppose $\psi\in H^\infty$ and $\varphi$ is an irrational rotation. Let $v$ be the outer part of $\psi$ and
$$r_{\psi,\varphi}=\lim_{k\to\infty}\inf_{m\in\mathfrak{M}_{L^\infty}}|\psi_{(k)}(m)|^{1/k}.$$
If $r_{\psi,\varphi}$ is less than $|v(0)|$, then for any $\epsilon>0$ and $n\in\mathbb{N}$, there exist $m_0\in\mathfrak{M}_{L^\infty}$ and $n'>n$ such that
$$|\psi_{(n')}(m_0)|^{1/n'}>|v(0)|-\epsilon$$
and
$$|\psi_{(n)}\comp\varphi_{-n}(m_0)|^{1/n}<r_{\psi, \varphi}+\epsilon.$$
\end{lemma}

\begin{proof}
The proof is similar with the proof of Lemma \ref{ts1}.

In fact, according to Remark \ref{mu}, the Ergodic Theorem shows that the equation
\begin{align}\label{jjj}
\lim_{k\to\infty}\left|\psi_{(k)}(m)\right|^{1/k}=|v(0)|
\end{align}
holds for $\mu_0$-almost every $m$ in $\mathfrak{M}_{L^\infty}$. For any $n\in\mathbb{N}$ and $\epsilon>0$, let
$$J=\{m\in\mathfrak{M}_{L^\infty} : |\psi_{(n)}(m)|<(r_{\psi, \varphi}+\epsilon)^n\}.$$
Then $J$ is non-empty. Moreover, since
$$\log|v_{(n)}(w)|=\int_{\mathfrak{M}_{L^\infty}}\log|\psi_{(n)}|\cdot P_wd\mu_0$$
for $w\in D$, where $P_w\in L^\infty(\partial D)$ is the Poisson kernel for the point $w$, so the fact
$$\inf_{w\in D}|v_{(n)}(w)|^{1/n}=\inf_{m\in\mathfrak{M}_{L^\infty}}|\psi_{(n)}(m)|^{1/n}\leqslant r_{\psi,\varphi}$$
implies that $\mu_0(J)>0$. Therefore we can take $m_1\in J$ such that (\ref{jjj}) holds for $\varphi_{n}(m_1)$. Then take $n'$ large enough, and $m_0=\varphi_{n}(m_1)$ is the point we want.
\end{proof}

Note that in the previous lemma, if the function $\psi$ has zeros in $\mathfrak{M}_{L^\infty}$, then $r_{\psi,\varphi}=0$. Otherwise, if $\psi$ has no zero in $\mathfrak{M}_{L^\infty}$, then its outer part $v$ belongs to $(H^\infty)^{-1}$, so we have $r_{\psi,\varphi}=\rho_{\frac{1}{v},\varphi}^{-1}$.

\begin{theorem}\label{noninver}
Suppose $\psi\in H^\infty$ and $\varphi$ is an irrational rotation. Let $v$ be the outer part of $\psi$ and
$$r_{\psi,\varphi}=\lim_{k\to\infty}\inf_{m\in\mathfrak{M}_{L^\infty}}|\psi_{(k)}(m)|^{1/k}.$$
If $r_{\psi,\varphi}$ is less than $|v(0)|$, then
$$\{\lambda\in\mathbb{C} : r_{\psi,\varphi}\leqslant|\lambda|\leqslant|v(0)|\}$$
is contained in the essential spectrum of $C_{\psi,\varphi}$.
\end{theorem}

\begin{proof}
Fixed an arbitrary positive number $\lambda$ such that $r_{\psi,\varphi}<\lambda<|v(0)|$ and $\lambda\ne|\psi(0)|$. Take $q_1,q_2>0$ satisfying that $r_{\psi,\varphi}<q_1<\lambda<q_2<|v(0)|$. Then for each $k\in\mathbb{N}$, by Lemma \ref{ts2} and using the same method as in the proof of Theorem \ref{inver}, one can find $n_k>k$ and $g_k\in A^p_\alpha$ such that
$$\left|\left|\psi_{(n_k)}g_k\right|\right|>q_2^{n_k}||g_k||$$
and
$$\left|\left|\psi_{(n_k+k)}\cdot g_k\comp\varphi_{k}\right|\right|<q_1^{k}\left|\left|\psi_{(n_k)}g_k\right|\right|.$$
Now repeat the proof of Theorem \ref{inver} we can see that there exist $\lambda_0$ such that $|\lambda_0|=\lambda$ and $C_{\psi,\varphi}-\lambda_0$ is not bounded from below. In the last part of the proof of Theorem \ref{inver} we have shown that each possible eigenvalue of $C_{\psi,\varphi}$ must have the same modulus with $\psi(0)$. Since $\lambda\ne|\psi(0)|$, $\lambda_0$ can not be a eigenvalue of $C_{\psi,\varphi}$, hence $C_{\psi,\varphi}-\lambda_0$ has no closed range. This means that $\lambda_0\in\sigma_{e}(C_{\psi,\varphi})$. Finally, by Lemma \ref{ri} and the fact that $\sigma_{e}(C_{\psi,\varphi})$ is closed, we get our conclusion.
\end{proof}

\begin{corollary}\label{main2}
Suppose $\psi\in H^\infty$ and $\varphi$ is an elliptic automorphism of order $\infty$. If $\psi$ has zeros in $\mathfrak{M}_{L^\infty}$, then on $A_\alpha^p$ one has
$$\sigma(C_{\psi,\varphi})=\sigma_e(C_{\psi,\varphi})=\{\lambda\in\mathbb{C} : |\lambda|\leqslant \rho_{\psi,\varphi}\}.$$
\end{corollary}

\begin{proof}
This corollary is a combination of Theorem \ref{inver} and Theorem \ref{noninver} as soon as one notice that $r_{\psi,\varphi}=0$ in this situation.
\end{proof}

The previous corollary gives a complete description of the spectrum and essential spectrum of $C_{\psi,\varphi}$ when $\psi$ has zeros in $\mathfrak{M}_{L^\infty}$, or equivalently, the outer part of $\psi$ is not bounded from below on $D$.

On the other hand, if all the zeros of $\psi$ lie in $D$, then the outer part of $\psi$ now belongs to $(H^\infty)^{-1}$, and the inner part of $\psi$ is a finite Blaschke product. The next Theorem shows that in this case $\sigma_e(C_{\psi,\varphi})$ no longer coincides with $\sigma(C_{\psi,\varphi})$

\begin{theorem}\label{main3}
Suppose $\psi\in H^\infty$ and $\varphi$ is an elliptic automorphism of order $\infty$. If $\psi$ has no zero in $\mathfrak{M}_{\infty}\backslash D$ but has zeros in $D$, then on $A_\alpha^p$ one has
$$\sigma_e(C_{\psi,\varphi})=\{\lambda\in\mathbb{C} : \rho_{\frac{1}{v},\varphi}^{-1}\leqslant|\lambda|\leqslant \rho_{\psi,\varphi}\}$$
and
$$\sigma(C_{\psi,\varphi})=\{\lambda\in\mathbb{C} : |\lambda|\leqslant \rho_{\psi,\varphi}\}.$$
\end{theorem}

\begin{proof}
For any fixed $\lambda$ with $|\lambda|<\rho_{\frac{1}{v},\varphi}^{-1}$, we will show that $C_{\psi,\varphi}-\lambda$ is bounded from below.

Take $q>0$ such that $|\lambda|<q<\rho_{\frac{1}{v},\varphi}^{-1}$. Assume that $\psi=\tau\cdot v$ where $\tau$ and $v$ are the inner and outer part of $\psi$ respectively. Since $\psi$ has no zero in $\mathfrak{M}_{\infty}$, $\tau$ is a finite Blaschke product and $v\in (H^\infty)^{-1}$. By the definition of $\rho_{\frac{1}{v},\varphi}$, for $n_0\in\mathbb{N}$ large enough we have
$$q^{n_0}<\inf_{m\in\mathfrak{M}_{L^\infty}}|\psi_{(n_0)}(m)|=\inf_{z\in D}|v_{(n_0)}(z)|.$$

For any $\epsilon>1$ there exist $R\in(0,1)$ such that $|\tau(z)|>1-\epsilon$ whenever $R<|z|<1$. Let $R'=\frac{R+1}{2}$. For any fixed $f\in A_\alpha^p $, write $\|f\|^p=I_1+I_2$, where
$$I_1=\frac{1+\alpha}{\pi}\int_0^{R}dr\int_0^{2\pi}|f(r\text{e}^{i\theta})|^p(1-r^2)^\alpha rd\theta$$
and
$$I_2=\frac{1+\alpha}{\pi}\int_R^{1}dr\int_0^{2\pi}|f(r\text{e}^{i\theta})|^p(1-r^2)^\alpha rd\theta.$$
Then
\begin{align*}
I_2&\geqslant\frac{1+\alpha}{\pi}\int_R^{R'}dr\int_0^{2\pi}|f(r\text{e}^{i\theta})|^p(1-r^2)^\alpha rd\theta\\
&\geqslant\frac{1+\alpha}{\pi}(1-R'^2)^\alpha R\int_R^{R'}dr\int_0^{2\pi}|f(r\text{e}^{i\theta})|^pd\theta\\
&\geqslant \frac{1+\alpha}{\pi}(1-R'^2)^\alpha (R'-R)\int_0^{R}dr\int_0^{2\pi}|f(r\text{e}^{i\theta})|^pd\theta\\
&\geqslant(1-R'^2)^\alpha (R'-R) I_1.
\end{align*}
So $I_2\geqslant\frac{C}{1+C}\|f\|^p$, where $C=(1-R'^2)^\alpha (R'-R)$. Therefore, for $n>n_0$ we have
\begin{align*}
\left|\left|C_{\psi,\varphi}^{n}f\right|\right|^p&\geqslant q^{np}\cdot\left|\left|\tau_{(n)}\cdot f\comp\varphi_n\right|\right|^p\\
&\geqslant q^{np}\cdot\frac{1+\alpha}{\pi}\int_R^{1}dr\int_0^{2\pi}|\tau_{(n)}(r\text{e}^{i\theta})f\comp\varphi_n(r\text{e}^{i\theta})|^p(1-r^2)^\alpha rd\theta\\
&\geqslant q^{np}(1-\epsilon)^{np}I_2\\
&\geqslant q^{np}(1-\epsilon)^{np}\frac{C}{1+C}\|f\|^p.
\end{align*}
By taking $\epsilon>0$ sufficiently small and $n\in\mathbb{N}$ large enough, we can make $q^{np}(1-\epsilon)^{np}\frac{C}{1+C}$ greater than $|\lambda|^{np}$. This means that $C_{\psi,\varphi}-\lambda$ is bounded from below.

Since $\tau$ is a finite Blaschke product, the codimension of the range of $C_{\psi,\varphi}$ is finite, Hence $C_{\psi,\varphi}$ is Fredholm. By the stability of the index of Fredholm operators, we can infer that $C_{\psi,\varphi}-\lambda$ is always a Fredholm operator whose index is not zero whenever $|\lambda|<\rho_{\frac{1}{v},\varphi}^{-1}$. Thus we have proven that
$$\{\lambda\in\mathbb{C} :  |\lambda|< \rho_{\frac{1}{v},\varphi}^{-1}\}\subset\sigma(C_{\psi,\varphi})\backslash\sigma_e(C_{\psi,\varphi}).$$
On the other hand, Theorem \ref{inver} and Theorem \ref{noninver} shows that
$$\{\lambda\in\mathbb{C} : \rho_{\frac{1}{v},\varphi}^{-1}\leqslant|\lambda|\leqslant \rho_{\psi,\varphi}\}\subset\sigma_e(C_{\psi,\varphi}).$$
\end{proof}

Finally there remains the case where $\psi$ has zeros in $\mathfrak{M}_{\infty}\backslash D$ but has no zero in $\mathfrak{M}_{L^\infty}$. We shall list this case as an open question here.

\begin{problem}
Suppose $\psi\in H^\infty$ and $\varphi$ is an irrational rotation. Does one have
$$\sigma_e(C_{\psi,\varphi})=\{\lambda\in\mathbb{C} : |\lambda|\leqslant\rho_{\psi,\varphi}\}$$
on $A_\alpha^p$ whenever $\psi$ has zeros in $\mathfrak{M}_{\infty}\backslash D$ but has no zero in $\mathfrak{M}_{L^\infty}$.

\end{problem}

Note that this will happen only when the outer part of $\psi$ is in $(H^\infty)^{-1}$ and the inner part of $\psi$ is not a finite Blaschke product. Theorem \ref{inver} and Theorem \ref{noninver} are still available in this situation, so we have
$$\{\lambda\in\mathbb{C} : \rho_{\frac{1}{v},\varphi}^{-1}\leqslant|\lambda|\leqslant \rho_{\psi,\varphi}\}\subset\sigma_e(C_{\psi,\varphi}).$$
Moreover, since $C_{\psi,\varphi}$ is not Fredholm in this case, we have $0\in\sigma_e(C_{\psi,\varphi})$. In fact, when the outer part of  $\psi$ is in $(H^\infty)^{-1}$ , it is known that $C_{\psi,\varphi}$ has closed range if and only if the inner part of $\psi$ is a finite product of interpolating Blaschke products, if and only if $\psi$ does not vanish on any trivial Gleason part of $\mathfrak{M}_{\infty}$. See \cite{rec2,rec1}. So if all the zeros of $\psi$ lie in non-trivial Gleason parts, then some neighbourhood of $0$ is contained in $\sigma_e(C_{\psi,\varphi})\backslash \sigma_{ap}(C_{\psi,\varphi})$; otherwise, if $\psi$ take zeros on some trivial Gleason parts, then $0\in\sigma_e(C_{\psi,\varphi})\cap\sigma_{ap}(C_{\psi,\varphi})$. Therefore, the structure of the (essential) spectrum of $C_{\psi,\varphi}$ in these situations can be expected to rely much on the location of the zeros of $\psi$ as well.

\end{document}